\documentclass{amsart}
\usepackage[T1]{fontenc}
\usepackage{amsfonts}
\usepackage{amsthm}
\usepackage{graphicx}
\usepackage{xcolor}
\usepackage{comment}
\usepackage{amssymb}
\usepackage{bbm}
\usepackage{cite}

\usepackage[
	colorlinks,
	linkcolor={black},
	citecolor={black},
	urlcolor={black}
]{hyperref}

\newcommand{\bbC}{{\mathbbm{C}}}
\newcommand{\bbN}{{\mathbbm{N}}}
\newcommand{\bbR}{{\mathbbm{R}}}
\newcommand{\bbZ}{{\mathbbm{Z}}}

\newcommand{\calA}{{\mathcal{A}}}
\newcommand{\calZ}{{\mathcal{Z}}}

\DeclareMathOperator{\tr}{tr}
\newcommand{\FVI}{{\operatorname{FVI}}}

\newtheorem{theorem}{Theorem}[section]
\newtheorem{prop}[theorem]{Proposition}

\theoremstyle{definition}
\newtheorem{remark}[theorem]{Remark}

\DeclareMathOperator{\spectrum}{spec}
\newcommand{\Hausdorff}{{\mathrm{H}}}
\newcommand{\loc}{{\mathrm{loc}}}
\newcommand{\UH}{{\mathcal{UH}}}
\newcommand{\idty}{\mathbbm{1}}
\newcommand{\strongstab}{{\mathrm{ss}}}
\newcommand{\centstab}{{\mathrm{cs}}}

\numberwithin{equation}{section}

\author[J. Fillman \and A. Luna]{Jake Fillman \and Alexandro Luna}

\title{On the  Spectra of Sieved Schr\"odinger Operators}

\begin{document}

\begin{abstract}
We give a family of examples of discrete Schr\"odinger operators whose spectral dimension  is not invariant under sieving.
The examples are produced from the Fibonacci Hamiltonian, which is one of the main models of a one-dimensional quasicrystal.
We also give a family of examples in which the local Hausdorff dimension tends to zero in some parts of the spectrum as the sieving parameter is sent to infinity.
\end{abstract}

\maketitle

\hypersetup{linkcolor={black!30!blue}, citecolor={black!30!green},urlcolor={black!30!blue}}

\section{Introduction}
\subsection{Setting and Main Results}

We discuss discrete Schr\"odinger operators $H = \Delta + V$ acting in $\ell^2(\bbZ)$ via
\begin{equation}
    [H\psi](n) = \psi(n-1) + \psi(n+1) + V(n)\psi(n),
\end{equation}
where the potential $V$ is bounded and real-valued. Such operators arise naturally in a tight-binding approximation and play an important role in spectral theory and mathematical physics.

A  recent work \cite{FGH2025} has discussed the impact of \emph{sieving} on certain almost-periodic Schr\"odinger operators.
More precisely, given $\ell \in \bbN :=\{1,2,3,\ldots\}$, the $\ell$-fold \emph{sieve} of $V$ is defined by
\begin{equation}
    V^{[\ell]}(n) =
    \begin{cases}
        V(n/\ell) & n \in \ell \bbZ \\
        0 & n \notin \ell \bbZ.
    \end{cases}
\end{equation}
We also abbreviate $H^{[\ell]} = \Delta+V^{[\ell]}$. An interesting question is how and whether the spectrum and spectral type of $H^{[\ell]}$ relates to that of $H$.

On one hand, there are a few positive results for Schr\"odinger operators. If $H$ is periodic, then so is $H^{[\ell]}$ and hence the dimension of the spectrum is independent of $\ell$, since the spectrum of a periodic Schr\"odinger operator is always a union of finitely many nondegenerate closed intervals and hence one-dimensional. 
A similar result can be established for random Schr\"odinger operators (for almost every realization).
Furthermore, the aforementioned work \cite{FGH2025} showed that for typical $V$ from a suitable class of almost-periodic sequences, the fractal dimension of $\spectrum H^{[\ell]}$ is zero for all $\ell \in \bbN$ \cite{FGH2025} (hence independent of $\ell$).
On the other hand, this question is well-motivated by the study of orthogonal polynomials on the unit circle and CMV (Cantero--Moral--Vel\'azquez) matrices; we direct the reader to the books \cite{Simon2005OPUC1, SimonOPUC2} for definitions and background information.
Indeed, for CMV matrices, it is known that sieving the operator coefficients produces a spectrum that is just a union of $\ell$ rescaled copies of the original spectrum and a similar statement holds for the spectral measure;
compare \cite[Example~1.6.14]{Simon2005OPUC1}.
In particular, for CMV matrices, both the spectral type and the dimension of the spectrum are always invariant under the sieving operation.

On account of the discussion in the previous paragraph, it would be quite natural to guess that the dimension of the spectrum of a Schr\"odinger operator is invariant under sieving; indeed, one of the authors was asked exactly this question while presenting \cite{FGH2025}, which was then the inspiration for this investigation.  
Our main result shows that the spectral dimension need \emph{not} be invariant under sieving for Schr\"odinger operators.

\begin{theorem} \label{t:main}
    There exists a bounded potential $V:\bbZ \to \bbR$ such that 
    \begin{equation}  \label{eq:sieveSpecDim}
    \dim_{\Hausdorff}(\spectrum H) \neq \dim_{\Hausdorff}(\spectrum H^{[\ell]})
    \end{equation}
    for every $\ell \geq 2$.
\end{theorem}

\begin{remark}
\mbox{\,}
\begin{enumerate}
    \item To the best of the authors' knowledge, these are the first examples of Schr\"odinger operators exhibiting this behavior.
    \item CMV matrices often play the role of unitary analogues of Schr\"odinger operators, and indeed, often results from one setting are carried over to the other. This is sometimes easy, sometimes difficult, and sometimes impossible. 
    As mentioned above, there cannot exist any example of a CMV matrix satisfying \eqref{eq:sieveSpecDim}. 
    Thus, Theorem~\ref{t:main} belongs to the comparatively rare class of results for which there does \emph{not} exist a bridge between the two worlds and one observes a genuinely new phenomenon.
    \item The reader will notice however that the ``bulk'' topological structure does not change in our examples. That is, in the examples we construct, $\spectrum H^{[\ell]}$ is a Cantor set for every $\ell$.
    We thus have the following question: is Cantor spectrum stable under sieving?
\end{enumerate}
\end{remark}

As $\ell \to \infty$, the potential energy exhibits long stretches on which it is constant, and hence locally indistinguishable from the free operator.
    On that basis, one might guess that the local fractal dimension increases to one throughout the spectrum.
    However, this need not be the case and indeed can fail in a rather dramatic fashion.
    
\begin{theorem} \label{t:main2}
 There exists $V:\bbZ \to \bbR$ bounded and $K \subseteq \bbR$ a compact interval such that $\dim_\Hausdorff( K \cap \spectrum H^{[\ell]})>0$ for every $\ell$, but
    \begin{equation}
        \lim_{\ell \to \infty} \dim_{\Hausdorff}(K \cap \spectrum H^{[\ell]})  = 0.
    \end{equation}
\end{theorem}

\begin{remark}\mbox{\,}
\label{rem:main}
    Both examples are produced by sieving the Fibonacci Hamiltonian, which is one of the most prominent models of a one-dimensional quasicrystal.
    The spectral type is purely singular continuous: the absence of absolutely continuous spectral measures follows because the spectrum has zero Lebesgue measure, while the absence of point spectrum follows from the Gordon lemma \cite{Gordon1976} (see \cite[Section~7.8]{DF2024ESO2} for a formulation suitable for the current setting).
    In this context, let us point out that it can already be seen from work on the Mosaic model that the spectral type of a Schr\"odinger operator may change upon sieving \cite{WangEtAl2023CMP}.
\end{remark}

\subsection{Results for the Sieved Fibonacci Hamiltonian}

Let us formulate the main results for the Fibonacci Hamiltonian more precisely.
Let $\calA = \{0,1\}$ and consider $\Omega \subseteq \{0,1\}^\bbZ$ the \emph{Fibonacci subshift}, which can be defined in a few equivalent ways; compare \cite[Chapter~10]{DF2024ESO2}. 
For us, this is most conveniently done through the Fibonacci substitution, which acts via $S:0 \mapsto 1$, $1 \mapsto 10$.
Iterating, one obtains a sequence of words of increasing length 
\begin{equation} \label{eq:wkdef} 
w_k :=S^k(1)
\end{equation}
such that $w_k$ is a prefix of $w_{k'}$ whenever $k \leq k'$ and hence there is a limiting (in the product topology sense)  word
\begin{equation}
    w_\infty = 101 10 101 101 10 101 10 101  \cdots
\end{equation}
which is a substitution word for $S$ in the sense that $S(w_\infty) = w_\infty$.
The Fibonacci subshift $\Omega$ then consists of all sequences $\omega = \{\omega(n) \}_{n \in \bbZ}$ such that for any $n \in \bbZ$ and $m \in \bbN$, the string $\omega(n) \omega(n+1) \cdots \omega(n+m-1)$ appears in $w_\infty$.
The \emph{shift map} $T:\Omega \to \Omega$ is given by $[T\omega](n) = \omega(n+1)$ for $n \in \bbZ$.

Given a choice of a coupling constant $\lambda>0$, a sieving parameter $\ell \in \bbN := \{1,2,3,\ldots\}$, and a choice of $\omega \in \Omega$, we obtain
\[ H_{\lambda,\omega}^{[\ell]} = \Delta + V_{\lambda,\omega}^{[\ell]}\]
where $V_{\lambda,\omega}(n) = \lambda \omega(n)$, that is
\begin{equation}
    V_{\lambda,\omega}^{[\ell]}(n)
    = \begin{cases}
        \lambda \omega(n/\ell) & n \in \ell \bbZ \\
        0 & n \notin \ell \bbZ.
    \end{cases}
\end{equation}

Given $\lambda$ and $\ell$, strong operator convergence and minimality of $(\Omega,T)$ (compare the proof of \cite[Theorem~4.9.1]{DF2022ESO1}) show that there is a fixed compact set $\Sigma_{\lambda}^{[\ell]} \subseteq \bbR$ such that
\begin{equation}
    \Sigma_\lambda^{[\ell]} = \spectrum H_{\lambda,\omega}^{[\ell]}
\end{equation}
for every $\omega \in \Omega$.

Our first main result characterizes the fractal properties of $\Sigma_\lambda^{[\ell]}$:

\begin{theorem} \label{t:sievedFibonacciDim1}
    For all $\lambda >0$ and $\ell \in \bbN$:
    \begin{enumerate}
        \item $\Sigma_{\lambda}^{[\ell]}$  is a Cantor set.
        \item $\Sigma_{\lambda}^{[\ell]}$  has Lebesgue measure zero.
        \item $\Sigma_{\lambda}^{[\ell]}$  has Hausdorff dimension one if $\ell \geq 2$ and Hausdorff dimension strictly less than one if $\ell =1$.
    \end{enumerate}
    In particular, $\Sigma_\lambda^{[\ell]}$ is not a dynamically defined Cantor set whenever $\ell \geq 2$.
    \end{theorem}
\begin{remark}\mbox{\,}
\begin{enumerate}
    \item Theorem~\ref{t:main} follows directly from Theorem~\ref{t:sievedFibonacciDim1} with $H = H_{\lambda,\omega}$ for any $\omega$ and $\lambda>0$.
\item  The results for $\ell \geq 2$ are new.
The case $\ell = 1$ corresponds to the standard Fibonacci Hamiltonian, which has been studied extensively.
    In particular, the results for $\ell=1$ are already known and only formulated for the sake of completeness. The zero-measure Cantor nature of the spectrum is due to S\"ut\H{o} \cite{Suto1987CMP, Suto1989JSP} (with an earlier result in \cite{Casdagli1986CMP}), while results about the dimension can be found in Damanik--Gorodetski--Yessen \cite{DamGorYes2016Invent}, with earlier partial results in \cite{Cantat2009, DG2011, DEGT2008}.
\end{enumerate}

\end{remark}

We are able to give some finer estimates on the fractal dimension.
    One of these we found to be quite striking and unexpected. Concretely, as $\ell \to \infty$, the potential energy exhibits long stretches on which it is constant, and hence locally indistinguishable from the free operator.
    On that basis, one might guess that the local fractal dimension increases to one throughout the spectrum, but quite the opposite turns out to be the case, at least when the coupling constant is not too small.

\begin{theorem}\label{t:local Hausdorff dimension tends to 0}
    For each $\lambda>4$, there is a compact interval $K \subseteq \bbR$ such that $\dim_\Hausdorff( K \cap \Sigma_\lambda^{[\ell]}) > 0$ for all $\ell$ and 
    \begin{equation}
        \lim_{\ell \to \infty} \dim_{\Hausdorff}(K \cap \Sigma_{\lambda}^{[\ell]})  = 0.
    \end{equation}
\end{theorem}

\subsection*{Acknowledgements} The authors thank Anton Gorodetski and Wencai Liu for helpful discussions and comments on an earlier version of the draft. 
J.F.\ was supported in part by National Science Foundation grant DMS-2513006.
A.L.\ was supported in part by National Science Foundation grant DMS-2247966 (PI: A.\ Gorodetski).

    \section{Trace Map Formalism}

The family of operators under consideration can be studied with a transfer matrix formalism similar to that of the standard Fibonacci Hamiltonian, after regrouping.
Similar constructions have appeared in many other places in the literature and have been successfully used to study quasiperiodic operators, e.g., \cite{DamFilGoh2022JST, WangEtAl2023CMP, ZhouEtAl2023PRL}.

To see how this comes about, we need to take advantage of the substitution structure of the Fibonacci subshift, which in turn produces a renormalization scheme for the transfer matrices.
First, observe that for any Schr\"odinger operator $H_V$, $H_Vu = Eu$ is equivalent to
\begin{equation}
    \begin{bmatrix}
        u(n+1) \\ u(n)
    \end{bmatrix}
    =  \begin{bmatrix}
        E -V(n) & -1 \\ 1 & 0
    \end{bmatrix}
    \begin{bmatrix}
        u(n) \\ u(n-1)
    \end{bmatrix} .
\end{equation}
Thus, for $n \in \bbN$, we define
\begin{equation}
    A_E^n(\lambda\omega^{[\ell]})
    = Y(E - \lambda\omega^{[\ell]}(n-1)) \cdots Y(E - \lambda\omega^{[\ell]}(1)) Y(E - \lambda\omega^{[\ell]}(0))
\end{equation}
where
\begin{equation}
    Y(s) = \begin{bmatrix}
        s & -1 \\ 1 & 0
    \end{bmatrix}, \quad s \in \bbC.
\end{equation}
The building blocks of the associated trace map scheme are the matrices
\begin{align}
M_{-1}(E,\lambda,\ell) & = \begin{bmatrix}
        1 & -\lambda \\ 0 & 1 
    \end{bmatrix}, \\
    M_0(E,\lambda,\ell) & =Y(E)^\ell, \\
    \label{eq:Mkrecursion}
    M_k(E,\lambda,\ell)  & = M_{k-2}(E,\lambda,\ell)M_{k-1}(E,\lambda,\ell), \quad k \in \bbN,
\end{align}
and the half-traces
\begin{equation}
    x_k(E,\lambda,\ell) = \frac12 \tr M_k(E,\lambda,\ell).
\end{equation}
Due to \eqref{eq:Mkrecursion}, note that
\begin{equation}
    M_1(E,\lambda,\ell)  = Y(E-\lambda)Y(E)^{\ell-1}
\end{equation}
The \emph{dynamical spectrum} is defined by 
\[B_\lambda^{[\ell]} = \{E : \{x_k(E,\lambda,\ell)\}_{k=1}^\infty \text{ is bounded}\}.\]

Choosing a substitution word $w_\infty$ of $S$ and an element $\omega_0 \in \Omega$ such that $\omega_0 |_{\bbZ_+} = w_\infty$ (where $\bbZ_+ := \bbN \cup\{0\}$), we see that
\begin{equation}
    A_E^{\ell F_k}(\lambda \omega_0^{[\ell]})
    = M_k(E,\lambda,\ell),
\end{equation}
for all $k \geq 1$, where $F_k$ denotes the $k$th Fibonacci number, normalized by $F_0=F_1=1$.

For later use, it is helpful for us to take advantage of the block structure of $\omega^{[\ell]}$ by decomposing into blocks of length $\ell$ in order to reduce to questions about a suitable cocycle defined over $(\Omega,T)$. 
Concretely, for any $n \in \bbN$, we may note that
\begin{equation}
    A_E^{\ell \cdot n}(\lambda \omega^{[\ell]})
    = \underbrace{\widetilde{A}_{E,\lambda,\ell}(T^{n-1}\omega) \cdots \widetilde{A}_{E,\lambda,\ell}(T\omega)\widetilde{A}_{E,\lambda,\ell}(\omega)}_{=:\widetilde{A}^n_{E,\lambda,\ell}(\omega)},
\end{equation}
where 
\begin{equation}
    \widetilde A_{E,\lambda,\ell}(\omega) := M_{\omega(0)} (E,\lambda, \ell),
\end{equation}
and with a similar definition for $n \leq -1$.

Another crucial quantity is given by the \emph{Lyapunov exponent}, which can be defined by
\begin{equation}
L(E)
=  L_\lambda^{[\ell]}(E)
= \lim_{n \to \infty} \frac{1}{n} \int_\Omega \! \log \| \widetilde A_E^n(\omega)\| \, d\mu(\omega),
\end{equation}
where $\mu$ denotes the unique $T$-invariant Borel probability measure on $\Omega$.
We denote $\calZ_\lambda^{[\ell]} = \{ E : L_\lambda^{[\ell]}(E) = 0\}$.
It will help us to prove a version of Johnson's theorem (cf.\ \cite{Johnson1986JDE}) adapted to the present setting.
To that end, let
\begin{equation}
    \UH_\lambda^{[\ell]} = \{E  : (T,\widetilde{A}_{E,\lambda, \ell}) \text{ is uniformly hyperbolic}\},
\end{equation}
where we recall that $(T,\widetilde{A}_{E,\lambda, \ell})$ is uniformly hyperbolic if and only if there is a constant $c>0$ such that
\begin{equation}
    \|\widetilde{A}_{E,\lambda,\ell}^n\| \geq ce^{c|n|} \quad \forall n \in \bbZ, \ \omega \in \Omega.
\end{equation}

We have the following result, which we will prove in Section~\ref{sec:cantor}.

\begin{theorem}\label{t:spectrum bounded orbit}
    For all $\lambda>0$ and $\ell \in \bbN$, $ \Sigma_\lambda^{[\ell]} = B_\lambda^{[\ell]} =\calZ_\lambda^{[\ell]} = \bbR \setminus \UH_\lambda^{[\ell]}$.
\end{theorem}

Let us give some computations that will be useful and necessary later.
By the  Cayley--Hamilton theorem,
\begin{equation}
    x_{k+1} = 2x_k x_{k-1} -  x_{k-2}, \quad k \in \bbN,
\end{equation}
so $(x_{k+1},x_k,x_{k-1}) = T(x_k,x_{k-1}, x_{k-2})$, where $T(x,y,z) = (2xy-z,x,y)$.
It is well-known that $T$ preserves the \emph{Fricke--Vogt invariant}:
\begin{equation}
    \FVI(x,y,z) = x^2+y^2+z^2-2xyz-1.
\end{equation}
The \emph{curve of initial conditions} is given by 
\begin{equation}
    \gamma(E) = \gamma(E,\lambda,\ell)
   : = (x_1(E,\lambda,\ell), x_0(E,\lambda,\ell), x_{-1}(E,\lambda,\ell)),
\end{equation}
and we abbreviate
\begin{equation}
    \FVI(E)
    = \FVI(E,\lambda,\ell)
    :=\FVI(\gamma(E, \lambda, \ell))
\end{equation}
in this context.
For later use, let us compute the curve of initial conditions explicitly.

\begin{prop}\label{prop:x0x1FVIcalcs}
    Given $E \in \bbR$, write $E = \xi+\xi^{-1}$ with $|\xi| \geq 1$. We have
    \begin{align}
    \label{eq:x0formula}
        x_0(E,\lambda,\ell) & = \frac{1}{2}(\xi^\ell + \xi^{-\ell}), \\
        \label{eq:x1formula}
        x_1(E,\lambda,\ell) & = \frac{1}{2}(\xi^\ell + \xi^{-\ell}) - \frac{\lambda}{2} \frac{\xi^\ell - \xi^{-\ell}}{\xi-\xi^{-1}},\\
        \label{eq:FVIformula}
        \FVI(E,\lambda, \ell) & = \frac{\lambda^2}{4} \left[ \frac{\xi^\ell - \xi^{-\ell}}{\xi-\xi^{-1}}\right]^2.
    \end{align}
\end{prop}

\begin{proof}
Begin by observing that $Y(E)$ has eigenvalues $\xi^{\pm 1}$ with eigenvectors
\[ v_\pm(\xi) = \begin{bmatrix}
    \xi^{\pm 1} \\ 1
\end{bmatrix}, \]
so \eqref{eq:x0formula} follows directly. Moreover, this leads to
\begin{align}
    Y(E)
    & = \frac{1}{\xi-\xi^{-1}}
    \begin{bmatrix}
        \xi & \xi^{-1} \\ 1 & 1
    \end{bmatrix}
    \begin{bmatrix}
        \xi & 0 \\ 0 & \xi^{-1}
    \end{bmatrix}
    \begin{bmatrix}
        1 & -\xi^{-1} \\ -1 & \xi
    \end{bmatrix}.
\end{align}

Combining this with \eqref{eq:Mkrecursion} gives
\begin{align}
    M_1
    & = M_{-1}M_0 \\
    & = 
        \frac{1}{\xi-\xi^{-1}}
        \begin{bmatrix}
        1 & -\lambda \\ 0 & 1
    \end{bmatrix}
    \begin{bmatrix}
        \xi & \xi^{-1} \\ 1 & 1
    \end{bmatrix}
    \begin{bmatrix}
        \xi^\ell & 0 \\ 0 & \xi^{-\ell}
    \end{bmatrix}
    \begin{bmatrix}
        1 & -\xi^{-1} \\ -1 & \xi
    \end{bmatrix},
\end{align}
which proves \eqref{eq:x1formula} after a short computation.
Using $x_{-1} \equiv 1$ and the definition of the Fricke--Vogt invariant, we deduce
\begin{align}
    \FVI(\gamma(E,\lambda,\ell)) 
     = (x_0-x_1)^2
     = \frac{\lambda^2}{4} \left[ \frac{\xi^\ell - \xi^{-\ell}}{\xi-\xi^{-1}} \right]^2,
\end{align}
as desired.
\end{proof}

\section{Cantor Spectrum} \label{sec:cantor}

In this section, we prove that the spectrum is a zero-measure Cantor set and that it coincides with the dynamical spectrum associated with the trace map, that is, the set of $E \in \bbR$ for which $T^k (\gamma(E))$ is bounded in $\bbR^3$.
Here, it is helpful to note that $\{H_{\lambda, \omega}^{[\ell]} : \omega \in \Omega\}$ can be realized as a subset of a family of ergodic operators with base dynamics $(\Omega_\ell,T_\ell)$ given by the (discrete) suspension of $\Omega$ with roof function identically equal to $\ell$, that is:
$\Omega_\ell = \Omega\times \{1,2,\ldots,\ell\}$, $T_\ell(\omega,j) = (\omega,j+1)$ for $j < \ell$ and $T_\ell(\omega,\ell) = (T\omega,1)$.
Consequently, we can freely use some of the standard results for ergodic operators \cite{DF2022ESO1, DF2024ESO2}.

\begin{proof}[Proof of Theorem~\ref{t:spectrum bounded orbit}]
Fix $\lambda$ and $\ell$, and drop them from the notation. We will prove a sequence of inclusions.

{\bf \boldmath (i) $\Sigma \subseteq B$.}
For each $k \in \bbZ_+$, $\{E : x_k(E) \in [-1,1]\} =: \Sigma_k$ is the spectrum of an $\ell F_k$-periodic Schr\"odinger operator with potential obtained by repeating the string $w_k$ periodically and sieving (see \eqref{eq:wkdef}).
These periodic operators can be chosen in such a way that they converge strongly to $H_{\omega_\star}$ for some $\omega_\star \in \Omega$, which yields (e.g.\ by \cite[Corollary~1.4.22]{DF2022ESO1}) that
\begin{equation}
    \Sigma = \spectrum H_{\omega_\star} \subseteq \bigcap_{N\geq 1} \overline{\bigcup_{k \geq N} \Sigma_k}.
\end{equation}
Since $x_{-1} \equiv 1$, we deduce from standard escape results (e.g.\ \cite[Theorem~10.5.4]{DF2024ESO2}) for the trace map that $\{x_k\}$ is unbounded if and only if for some $k \geq 0$, $|x_k|>1$ and $|x_{k+1}|>1$; furthermore, in this case $|x_k|$ diverges to $\infty$ superexponentially quickly. 
Consequently,
\begin{equation}
    \bigcap_{N\geq 1} \overline{\bigcup_{k \geq N} \Sigma_k}
= \bigcap_{k = 0}^\infty (\Sigma_k \cup \Sigma_{k+1}) 
= B,
\end{equation}
concluding the proof of this inclusion.
\medskip

{ \bf \boldmath (ii) $B \subseteq \calZ$.} Arguing by contradiction, assume that there exists $E \in B\setminus \calZ$.
Since $L(E)>0$, the subadditive ergodic theorem implies
\begin{equation}
    \lim_{n\to\infty} \frac{1}{n} \log\|\widetilde A_E^n (\omega)\| = L(E)>0
\end{equation}
for $\mu$-a.e.\ $\omega \in \Omega$. Fixing such an $\omega$, there exists a unit vector $u$ such that $\|\widetilde A_E^n(\omega)u \|$ decays exponentially as $n \to \infty$.
However, due to $E \in B$ and \cite[Theorem~10.4.6]{DF2024ESO2}, there exist $\ell_k \to \infty$ such that $\omega(j) = \omega(j+\ell_k)$ for all $0 \le j < \ell_k$ and $\{\tr\widetilde A_E^{\ell_k}(\omega)\}$ is bounded as $k \to \infty$.
On account of the two-block Gordon lemma (e.g.\ \cite[Lemma~7.8.5]{DF2024ESO2}), $\|\widetilde A_E^n(\omega)u \|$ cannot tend to zero as $n \to \infty$, a contradiction.
\medskip

{\bf \boldmath (iii) $\calZ \subseteq \bbR \setminus \UH$.} This is immediate from the definitions.\medskip

{\bf \boldmath (iv)  $\bbR \setminus \UH \subseteq \Sigma$.} If $E \in \bbR \setminus \UH$, then by standard characterizations of uniform hyperbolicity (e.g.\ \cite[Theorem~3.8.2]{DF2022ESO1}) there exist $\omega_* \in \Omega$ and a unit vector $u = (u(0), u(-1))^\top \in \bbR^2$ such that
\begin{equation}
    \|\widetilde A^n_E(\omega_*) u\| \leq 1 \quad \forall n \in \bbZ.
\end{equation}
This implies that $u$ extends to a solution $u:\bbZ \to \bbR$ of $Hu = Eu$ that is bounded, and hence $E$ is a generalized eigenvalue of $H_{\omega_*}$, leading to 
\begin{equation}
    E \in \sigma(H_{\omega_*})
\end{equation}
Combining (i)--(iv) concludes the proof.
\end{proof}

\begin{remark}
    An alternate approach to proving $\Sigma = \calZ = \bbR \setminus \UH$ would be to directly prove $\Sigma = \bbR \setminus \UH$ (which follows from a slight modification of Johnson's theorem \cite{Johnson1986JDE}) and then to apply the Boshernitzan criterion \cite{Bosh1992ETDS} as in \cite{DamLen2006Duke} to deduce the absence of non-uniform hyperbolicity, which proves $\Sigma = \calZ$.
    We use the approach here in order to also deduce $\Sigma = B$, which is needed for the use of the hyperbolic theory later on.
\end{remark}

\begin{proof}[{\bf Proof of Theorem~\ref{t:sievedFibonacciDim1}, parts (1) and (2)}]
    As before, fix $\lambda$ and $\ell$ and suppress them from the notation.
    By a result of Pastur \cite{Pastur1980CMP}, the spectrum has no isolated points, so it suffices to establish (2).
    From Theorem~\ref{t:spectrum bounded orbit}, we know $\Sigma = \calZ$, so it suffices to show that the latter has zero Lebesgue measure. 
    However, by a result of Kotani \cite{Kotani1989RMP} $\calZ$ cannot have positive measure, since the potentials in question assume finitely many values and are not periodic.
\end{proof}

\section{Hyperbolicity for the Trace Map and Consequences}

Recall that the trace map $T:\bbR^3\rightarrow \bbR^3$ is given by $T(x,y,z):=(2xy-z,x,y)$. Since $T$ preserves the Fricke--Vogt invariant, the surface
$$ S_V:=\left\{(x,y,z)\in\bbR^3:\FVI(x,y,z) = V \right\}$$
is $T$-invariant for every $V \geq 0$. 
For such $V$, let us define
\begin{equation}
    \Omega_{V}:=\left\{x\in S_{V} : \left(T^k(x)\right)_{k\in\bbZ} \ \text{is bounded} \right\}.
\end{equation}

Throughout this section, we freely use the terminology of hyperbolic dynamics; for additional background, we direct the reader to \cite[Chapter~6]{KH1995}. 

\begin{theorem}[Cantat (2009) \cite{Cantat2009}]
    For each $V>0$, the set $\Omega_{V}$ is a locally maximal hyperbolic set for $T\restriction S_{V}$ that is homeomorphic to a Cantor set.
\end{theorem}

Through the stable manifold theory \cite{HP1970}, we know that for each $V\geq 0$ and each $p\in\Omega_V$,  the set 
$$W^{\strongstab}(p):=\left\{q : \lim\limits_{n\rightarrow\infty}\left|T^n(p)-T^n(q)\right|=0\right\}$$
is a smooth injectively immersed curve contained in $S_{V}$.
This curve is referred to as the \textit{strong stable manifold} at $p$. For $V\geq 0$, let us set 
\begin{equation}
    \Omega^+_{V}
    :=\left\{p\in S_{V} : \left(T^k(p)\right)_{k\in\bbN} \ \text{is bounded} \right\},
\end{equation}
and 
\begin{equation}
    \Omega^+:=\bigcup_{V> 0}\Omega_{V}^+ \ \text{and} \ \Omega_*^+:=\Omega^+\cup\Omega_0^+.
\end{equation}
Since $I(E,\lambda, \ell) \geq 0$ for all relevant $E$, $\lambda$, and $\ell$ (by Proposition~\ref{prop:x0x1FVIcalcs}) part of the conclusion of Theorem~\ref{t:spectrum bounded orbit} can be rephrased as follows:

\begin{theorem}\label{t: curve and bounded orbit equivalence}
    For each $\lambda$ and $\ell$, we have 
    $$E\in\Sigma_{\lambda}^{[\ell]} \Longleftrightarrow \gamma(E,\lambda,\ell)\in \Omega^+_*.$$
\end{theorem}

Through the use of center stable manifold theory, the following was deduced in \cite{DMY2013}.

\begin{prop}[Damanik--Munger--Yessen (2013) {\cite[Theorem 2.6]{DMY2013}}]
    There is a family of two-dimensional pairwise disjoint submanifolds of $\bbR^3$, denoted by $\mathcal W^\centstab$, such that 
    \begin{itemize}
        \item Each $W\in \mathcal W^\centstab$ is $T$-invariant.
        \item For each $p\in \Omega$, there is a unique $W(p)\in \mathcal W^\centstab$ such that $p\in W(p)$, and conversely, for each $W\in \mathcal W^\centstab$, $W\cap \Omega\neq\emptyset$.
        \item For each $V>0$ and $W\in \mathcal W^\centstab$, we have $W\cap S_{V}=W^{\strongstab}(p)$ for some $p\in \Omega_{V}$.
        \item For each $V>0$ and $W\in \mathcal W^\centstab$, $ W\cap S_{V}\neq \emptyset$ and this intersection is transverse.

        \item $\Omega^+=\bigsqcup_{W\in\mathcal W^\centstab} W$
    \end{itemize}
\end{prop}
The manifolds in $W^\centstab$ are \textit{center stable manifolds} for $T$ (see Section~2.2 and Appendix~A of \cite{DMY2013} for a formal discussion in the trace map context). 
On account of part~(1) of Theorem~\ref{t:sievedFibonacciDim1}, we are able to achieve:

\begin{prop}\label{p: local dimension is half of non wandering dimension}
    If $E\in \Sigma_{\lambda}^{[\ell]}$ and $\gamma(E, \lambda, \ell)\in S_{V_0}$, then if $V_0>0$, we have
    \begin{equation}\label{e: local dimension is half of non wandering dimension}
\dim_{\Hausdorff}^\loc(\Sigma_{\lambda}^{[\ell]},E)=\frac{1}{2}\dim_{\Hausdorff} \Omega_{V_0}.
    \end{equation}
    If $V_0=0$, then 
    \begin{equation}\label{e: local Hausdorff dimension is 1 for points on cayley cubic}
\dim_{\Hausdorff}^\loc(\Sigma_{\lambda}^{[\ell]},E)=1.
    \end{equation}
\end{prop}
\begin{proof}
By Theorem~\ref{t: curve and bounded orbit equivalence}, the spectrum $\Sigma_{\lambda}^{[\ell]}$ is diffeomorphic to $\gamma^\star\cap \Omega^+_*$, where 
\[ \gamma^\star:=
\{ \gamma(E,\lambda, \ell) : E \in I\}\]
and $I$ is the convex hull of $\Sigma_{\lambda}^{[\ell]}$. Here, $\gamma^\star$ is a compact analytic curve, and by part~(1) of Theorem~\ref{t:sievedFibonacciDim1}, its intersection with $\Omega^+_*$ is a Cantor set with respect to the subspace topology on $\gamma^\star$. In particular, $\gamma^\star\cap \Omega^+_*$ contains no isolated points and no intervals. 

From the previous proposition, this means that $\gamma^\star$ does not lie inside of any single center stable manifold $W$ from $\mathcal W^{\centstab}$ and it is not completely contained in $S_0$. So, from \cite[Lemma~2.9]{DMY2013}, for any non-empty neighborhood $U\subset \gamma^\star\cap \Omega^+$, there is a point of transverse intersection with some center stable manifold $W$. Now, \cite[Theorem~2.13]{DMY2013} immediately implies (\ref{e: local dimension is half of non wandering dimension}), and \cite[Theorem~2.17]{DMY2013} immediately implies (\ref{e: local Hausdorff dimension is 1 for points on cayley cubic}).
\end{proof}

\begin{remark}
    It is actually the case that for each $V\geq 0$,
    \begin{equation}
d^s_{V}=d^u_{V}=\frac{1}{2}\dim_{\Hausdorff} \Omega_{V},
\end{equation}
where $d^s_{V}$ (resp.\ $d^u_{V}$) is the stable (resp.\ unstable) dimension of $\Omega_{V}$ (see Chapter 9 of \cite{B2012} for a detailed discussion of (un)stable dimension of hyperbolic sets). 
That these dimensions coincide follows from the dimension formulas given in \cite{MM1983}, the fact that $T\restriction_{S_{V}}$ is area preserving, and compactness of $\Omega_V$.
\end{remark}

From \cite{P2015, DG2011, DEGT2008}, we have the following useful facts.
\begin{prop}\label{p:Asymptotics of dimension of non wandering}
    The map $F:[0,\infty)\rightarrow [0,1]$ defined by $V\mapsto \frac{1}{2}\dim_{\Hausdorff}\Omega_V$ satisfies the following:
    \begin{itemize}
    \item F is real analytic on $(0,\infty)$
        \item $F$ is continuous at $0$ and $F(0)=1$
        \item $F(V)\in(0,1)$ for all $V>0$
        \item $\lim\limits_{V\rightarrow \infty} F(V)=0$ and in particular, 
        \begin{equation} 
        \lim\limits_{V\rightarrow \infty} F(V)\log V = \log(1+\sqrt{2}).
        \end{equation}
    \end{itemize}
\end{prop}

We can now use hyperbolicity and its consequences to prove all of the main results.
\begin{proof}[{\bf Proof of Theorem~\ref{t:sievedFibonacciDim1}, part (3).}]
    Let $\lambda>0$ and $\ell\geq 2$. 
    We show that there for some $E\in \Sigma_{\lambda}^{[\ell]}$, we have $\gamma(E,\lambda,\ell)\in S_0$, so result will follow from Proposition~\ref{p: local dimension is half of non wandering dimension}. 
    
Consider $\theta = \pi j/\ell$ for an integer $1 \le j \le \ell-1$ and $E = 2\cos\theta$.
In view of Proposition~\ref{prop:x0x1FVIcalcs}, we see that 
\begin{equation} 
\FVI(E,\lambda , \ell) = \frac{\lambda^2}{4} \frac{\sin^2\ell\theta}{\sin^2\theta} = 0,
\end{equation}
so it suffices to show that $E \in \Sigma_\lambda^{[\ell]}$.
    Here, we observe that $M_0(E,\lambda,\ell)=\idty_{2\times 2}$ and
    \begin{equation}
        x_1(E,\lambda,\ell) = \cos\ell\theta-\frac{\lambda}{2}\frac{\sin\ell\theta}{\sin\theta} = \cos\ell\theta = \pm 1.
    \end{equation}
    Consequently, for any $\omega \in \Omega$, there is a solution of $H_\omega u = Eu$ that is bounded, and therefore $E \in \Sigma_\lambda^{[\ell]}$ as desired.
\end{proof}

\begin{proof}[{\bf Proof of Theorem~\ref{t:local Hausdorff dimension tends to 0}}]
    Let $\lambda>4$ be given, and denote $K = [\lambda-2,\lambda+2]$. 
    By standard perturbative estimates (e.g.\ \cite[Lemma~5.2.4]{DF2024ESO2}), it follows that
    \begin{equation}
        K \cap \Sigma_\lambda^{[\ell]} \neq \emptyset
    \end{equation}
    for every $\ell \in \bbN$.
    On the other hand, one can readily check that $I(E,\lambda,\ell) \to \infty$ uniformly on $K$ as $\ell \to\infty$. 
    Indeed, writing $E>2$ as $E = 2\cosh\theta$ with $\theta>0$, we have for $E \in K$:
    \begin{align*}
        I(E,\lambda,\ell)
         = \frac{\lambda^2}{4} \frac{\sinh^2(\ell\theta)}{\sinh^2(\theta)} 
         \geq \frac{\lambda^2}{4} \frac{\sinh^2(\ell \operatorname{arccosh}(\frac{\lambda}{2}-1))}{\sinh^2( \operatorname{arccosh}(\frac{\lambda}{2} + 1))},
    \end{align*}
    which goes to $\infty$ as $\ell \to \infty$.
    As before then, the theorem follows from Proposition~\ref{p: local dimension is half of non wandering dimension} and Proposition ~\ref{p:Asymptotics of dimension of non wandering}. 
\end{proof}

\begin{proof}[{\bf Proof of Theorem~\ref{t:main}}]
   This follows directly from Theorem~\ref{t:sievedFibonacciDim1} with $H = H_{\lambda,\omega}$ for any $\omega$ and $\lambda>0$.
\end{proof}

\begin{proof}[{\bf Proof of Theorem~\ref{t:main2}}]
    This follows directly from Theorem~\ref{t:local Hausdorff dimension tends to 0}.
\end{proof}

\bibliographystyle{abbrv}
\bibliography{refs}

\end{document}